\documentclass[reqno]{amsproc}
\usepackage{amsthm}
\usepackage{amsfonts}
\usepackage{amsmath}
\usepackage{amssymb}
\usepackage{mathrsfs}
\usepackage{float}
\usepackage[active]{srcltx}
\usepackage[all]{xy}
\usepackage{hyperref}
\usepackage{url}

\makeatletter
\@namedef{subjclassname@2020}{%
  \textup{2020} Mathematics Subject Classification}
\makeatother

\newcommand{\ncom}{\newcommand}     
\newcommand*{\Le}{\leqslant}
\newcommand*{\Ge}{\geqslant}
\ncom{\beqn}{\begin{eqnarray*}}
\ncom{\eeqn}{\end{eqnarray*}}
\ncom{\beq}{\begin{eqnarray}}
\ncom{\eeq}{\end{eqnarray}}
\ncom{\inp}[2]{\langle{#1},\,{#2} \rangle}
\ncom{\rar}{\rightarrow}

\usepackage{color}
\newtheorem{theorem}{Theorem}[section]

\newtheorem{lemma}[theorem]{Lemma}

\newtheorem{proposition}[theorem]{Proposition}
\newtheorem{corollary}[theorem]{Corollary}
\theoremstyle{definition}
\newtheorem{definition}[theorem]{Definition}
\theoremstyle{remark}
\newtheorem{remark}[theorem]{Remark}
\usepackage[normalem]{ulem}
     
     \title[A transference principle]{A transference principle for involution-invariant \\functional Hilbert spaces}

\author[S. Bera, S. Chavan, S. Jain]{Santu Bera, Sameer Chavan, Shubham Jain}

\address{Department of Mathematics and Statistics\\
Indian Institute of Technology Kanpur, India}
\email{santu20@iitk.ac.in}
  \email{chavan@iitk.ac.in}
  
\address{Department of Mathematics, Indian Institute of Technology Guwahati, Guwahati, 781039, India}
   \email{shubjainiitg@iitg.ac.in}

\thanks{The first author is supported through the PMRF Scheme (2301352).}



\begin{document}

\begin{abstract}  
Let $\sigma : \mathbb C^d \rar \mathbb C^d$ be an affine-linear involution such that $J_\sigma = -1$ and let $U, V$ be two domains in $\mathbb C^d.$ Let $\phi : U \rar V$ be a $\sigma$-invariant $2$-proper map such that $J_\phi$ is affine-linear and  
let $\mathscr H(U)$ be a $\sigma$-invariant reproducing kernel Hilbert space of complex-valued holomorphic functions on $U.$  It is shown that the space 
$\mathscr H_\phi(V):=\{f \in \mathrm{Hol}(V) : J_\phi \cdot f \circ \phi \in  \mathscr H(U)\}$
endowed with the norm 
$\|f\|_\phi :=\|J_\phi \cdot f \circ \phi\|_{\mathscr H(U)}$
is a reproducing kernel Hilbert space and the linear mapping $\varGamma_\phi$ defined by 
$
\varGamma_\phi(f) = J_\phi \cdot f \circ \phi,$ $f \in \mathrm{Hol}(V),$  
 is a unitary from $\mathscr H_\phi(V)$ onto $\{f \in \mathscr H(U) : f = -f \circ \sigma\}.$ Moreover, a neat formula for the reproducing kernel $\kappa_{\phi}$ of $\mathscr H_\phi(V)$ in terms of the reproducing kernel of $\mathscr H(U)$ is given.  
The above scheme is applicable to symmetrized bidisc, tetrablock, $d$-dimensional fat Hartogs triangle and $d$-dimensional egg domain. Although some of these are known, this allows us to obtain an analog of von Neumann's inequality for contractive tuples naturally associated with these domains.  
\end{abstract} 
\maketitle

\section{Introduction and statement of the main theorem}
In this paper, we address the problem of constructing a Hardy space of a given domain in the $d$-dimesional Euclidean space (see \cite{CJP2023, GGLV2021, MRZ2013, M2021}). In particular, we provide an alternate approach to this problem based on the theory of reproducing kernel Hilbert spaces. Apart from a motivation from the function-theory, the existence of a reasonable Hardy space (cf. conditions (a) and (b) on \cite[p.~2172]{GGLV2021}) allows us to obtain an analog of von Neumann's inequality for a class of operator tuples naturally associated with the domain in question (cf. \cite[Theorem~2.1]{CJP2022}). In simple words, the main result of this paper allows us to construct a Hardy space $H^2(\Omega_2)$ from a known Hardy space $H^2(\Omega_1)$ of a domain $\Omega_1$ in various situations in which $\Omega_2$ is a $2$-proper image of $\Omega_1.$ Our construction is based on a circle of ideas in \cite{BDGR2022, G2021, MRZ2013, T2013}, and it is so general that it also applies to Bergman spaces (see \cite{K2001}). Another important aspect of this principle is that it gives a transformation formula relating the reproducing kernels in question.   

For a set $A,$ $\mbox{card}(A)$ denotes the cardinality of $A.$  For a positive integer $d,$ let $\mathbb C^d$ denote the $d$-fold cartesian product of the complex plane $\mathbb C.$ The unit polydisc centred at the origin in $\mathbb C^d$ is denoted by $\mathbb D^d.$ A map $\psi : \mathbb C^d \rar \mathbb C^d$ is {\it affine-linear} if there exists a linear operator $A$ on $\mathbb C^d$ and $B \in \mathbb C^d$ such that $\psi(z)=Az+B,$ $z \in \mathbb C^d.$ Let $U$ be a domain in $\mathbb C^d,$ that is, a nonempty open connected subset of $\mathbb C^d$. The set of zeros of a function $f : U \rar \mathbb C$ is denoted by $Z(f).$ 
Let $V$ be a domain in $\mathbb C^d.$ 
The {\it Jacobian} of a holomorphic map $\phi : U \rar V$ is the determinant of the {\it Jacobian matrix} of $\phi$ and it is denoted by $J_\phi.$ A continuous map $\phi : U \rar V$ is {\it proper} if pre-image under
$\phi$ of any compact subset of $V$ is compact. 
Let $\phi : U \rar V$ be a proper holomorphic map. It is noticed in \cite[Section~1.2]{R1982} that there exists a positive integer $m$ (to be referred to as the {\it multiplicity} of $\phi$) such that  
\beq \label{admissible-prop}
&&\mbox{card}(\phi^{-1}(\{z\})) = m  ~\mbox{for every~} z \in V \backslash \phi(Z(J_\phi)), \notag\\ 
&& \mbox{card}(\phi^{-1}(\{z\})) < m  ~\mbox{for every~} z \in \phi(Z(J_\phi)).
 \eeq
For short, we refer to a proper map of multiplicity $m$ as {\it $m$-proper map}.
A continuous map $\sigma : \mathbb C^d \rar \mathbb C^d$  is said to be an {\it involution} if $\sigma \circ \sigma = I.$ 
We reserve the notation $\mathrm{Fix}(\sigma)$ for the set $\{z \in U : \sigma(z)=z\}$ of fixed points of $\sigma.$
We say that $U$ is {\it $\sigma$-invariant} if $\sigma(z) \in U$ for every $z \in U.$ The map $\phi : U \rightarrow V$ is said to be {\it $\sigma$-invariant} if $U$ is $\sigma$-invariant and $\phi \circ \sigma = \phi.$ 
\begin{remark} 
Let $U, V$ be domains in $\mathbb C^d$ and $\phi : U \rar V$ be a holomorphic $2$-proper map. It was shown in the proof of \cite[Proposition~2.3]{GZ2024} that there exists a holomorphic involution $\sigma : U \rar U$ such that
\beq \label{3-prop}
\mbox{$\sigma$ is not the identity map, $\phi \circ \sigma = \phi$ and $\mathrm{Fix}(\sigma)=Z(J_\phi).$}
\eeq
To see the uniqueness, let $\sigma_j,$ $j=1, 2$ be two holomorphic involutions satisfying \eqref{3-prop}. Let $w \in U \setminus Z(J_\phi),$ and note that $\sigma_j(w) \neq w,$ $j=1, 2.$ If $\phi(w) \in \phi(Z(J_\phi)),$ then $\phi(\sigma_j(w)) = \phi(w) \in \phi(Z(J_\phi)),$ and hence by \eqref{admissible-prop}, $\sigma_j(w)=w,$ $j=1, 2.$ This contradiction shows that $\phi(w) \notin \phi(Z(J_\phi)).$ However, $\phi$ being $2$-proper, $\phi(w)$ has $2$ preimages under $\phi.$ Consequently, $\sigma_1=\sigma_2.$
\end{remark}

 Let $U$ be a domain in $\mathbb C^d.$ The linear space of complex-valued holomorphic functions on $U$ is denoted by $\mathrm{Hol}(U).$ 
If $U$ is bounded, then we reserve the notation $H^{\infty}(U)$ for the Banach algebra of complex-valued bounded holomorphic functions on $U$ endowed with the supremum norm $\|\cdot\|_{\infty, U}.$  
By a {\it functional Hilbert space} $\mathscr H(U)$, we understand a reproducing kernel Hilbert space (for short, RKHS) of complex-valued holomorphic functions on $U$ (for the definition of RKHS, see \cite[Definition~1.1]{PR2016}). For any functional Hilbert space $\mathscr H(U),$ there exists a positive semi-definite kernel $\kappa : U \times U \rar \mathbb C$ such that  
\beq
\label{rproperty}
\inp{f}{\kappa(\cdot, w)}=f(w), \quad f \in \mathscr H(U), ~w \in U
\eeq
(refer to \cite{PR2016} for the basic theory of RKHS). A {\it multiplier of $\mathscr H(U)$} is a function $f : U \rar \mathbb C$ such that $f \cdot g \in \mathscr H(U)$ for every $g \in \mathscr H(U).$ The {\it multiplier algebra} $\mathcal M(\mathscr H(U))$ is the set of multipliers of $\mathscr H(U).$
By the closed graph theorem and  \eqref{rproperty}, any multiplier $f$ yields the operator $\mathscr M_f \in \mathcal B(\mathscr H(U))$ of multiplication by $f,$ where $\mathcal B(\mathcal H)$ denotes the $C^*$-algebra of bounded linear operators on a Hilbert space $\mathcal H.$ Moreover, the operator norm makes $\mathcal M(\mathscr H(U))$ a unital
subalgebra of $\mathcal B(\mathscr H(U))$ that is closed in the weak operator topology (see \cite[Corollary 5.24]{PR2016}).

We now define a class of functional Hilbert spaces central to this paper (cf. \cite[p.~753]{BGMR2019}).    
%\begin{definition} 
Let $U$ be a $\sigma$-invariant domain.   
A functional Hilbert space
$\mathscr H(U)$ is said to be {\it $\sigma$-invariant} if for every $f \in \mathscr H(U),$ $f \circ \sigma \in \mathscr H(U)$ and 
\beq \label{sigma-invariant}
\inp{f \circ \sigma}{g \circ \sigma} = \inp{f}{g}, \quad f, g \in \mathscr H(U).
\eeq
Define the 
 subspaces (cf. \cite[p.~2]{G2021}) $\mathscr H^{\sigma}_{+}(U)$ and $\mathscr H^{\sigma}_{-}(U)$ of $\mathscr H(U)$ by  
\beqn
\mathscr H^{\sigma}_{+}(U) := \{f \in \mathscr H(U) : f = f \circ \sigma\}, ~
\mathscr H^{\sigma}_{-}(U) := \{f \in \mathscr H(U) : f = -f \circ \sigma\}.
\eeqn
%\end{definition}
\begin{remark} 
%\label{rmk-kernel-plus}
Assume that $\mathscr H(U)$ is a $\sigma$-invariant functional Hilbert space. Then $\mathscr H^{\sigma}_+(U)$ and $\mathscr H^{\sigma}_-(U)$ are closed subspaces of $\mathscr H(U).$ This follows from the fact that the convergence of a sequence in an RKHS implies its pointwise convergence. 
%Thus, $\mathscr H^{\sigma}_+(U)$ and $\mathscr H^{\sigma}_-(U)$ are reproducing kernel Hilbert spaces. 
Moreover, $\mathscr H(U) = \mathscr H^{\sigma}_+(U) \oplus \mathscr H^{\sigma}_{-}(U)~\mathrm{(the~ orthogonal ~direct ~sum)}$ and the reproducing kernels $\kappa^{\sigma}_{+}$ and $\kappa^{\sigma}_{-}$ of $\mathscr H^{\sigma}_{+}(U)$ and $\mathscr H^{\sigma}_{-}(U),$ respectively, are given by
\beq
\label{kappa-s-minus}
 \left.
 \begin{array}{ccc}
\kappa^{\sigma}_{+}(z, w) &=&\frac{1}{2}\Big(\kappa(z, w) + \kappa(\sigma(z), w)\Big) \\
\kappa^{\sigma}_{-}(z, w) &=& \frac{1}{2}\Big(\kappa(z, w) - \kappa(\sigma(z), w)\Big) 
 \end{array}
\right\},\quad z, w \in U.
\eeq 
A general version of the decomposition above has been recorded in
\cite[p.~754]{BGMR2019}, while the ``moreover" part may be easily derived from \eqref{rproperty} and the fact that any RKHS admits a unique reproducing kernel.   
\end{remark}

The following is the main result of this paper (a number of similar instances of this have been discussed in the literature; \cite{BGMR2019, G2021, GR2024, MRZ2013, T2013}). 
\begin{theorem} \label{main-result} 
Let $\sigma : \mathbb C^d \rar \mathbb C^d$ be an affine-linear involution and let $U, V$ be two domains in $\mathbb C^d$. Let $\phi : U \rar V$ be a 
$\sigma$-invariant $2$-proper map and let $\mathscr H(U)$ be a 
%$\sigma$-invariant reproducing kernel Hilbert space of complex-valued holomorphic functions on $U$ 
$\sigma$-invariant functional Hilbert space with the reproducing kernel $\kappa : U \times U \rar \mathbb C.$ 
Assume further that $J_\sigma = -1$ and $J_\phi$ is affine-linear.
Then, the linear space \beq \label{H(V)}
\mathscr H_\phi(V):=\{f \in \mathrm{Hol}(V) : J_\phi \cdot f \circ \phi \in  \mathscr H(U)\} 
\eeq 
endowed with the norm
\beq \label{phi-norm}
\|f\|_\phi :=\|J_\phi \cdot f \circ \phi\|_{\mathscr H(U)} 
\eeq
is a Hilbert space. 
%$\phi=(\phi_1, \ldots, \phi_d),$ where $\phi_1, \ldots, \phi_{d-1}$ are homogenous polynomials of degree $1$ and $\phi_d$ is a homogenous polynomial of degree $2$ such that $\phi^{-1}(\{0\})=\{0\}$. 
Moreover, for the linear mapping $\varGamma_\phi$ defined by 
$
\varGamma_\phi(f) = J_\phi \cdot f \circ \phi,$ $f \in \mathrm{Hol}(V),$ we have the following statements$:$
\begin{enumerate}
\item[$(i)$]  $\varGamma_\phi$ 
 is a unitary from $\mathscr H_\phi(V)$ onto $\mathscr H^{\sigma}_{-}(U).$
 \item[$(ii)$]  $\mathscr H_\phi(V)$ is a reproducing kernel Hilbert space with the reproducing kernel $\kappa_{\phi}$ given by
\beq 
\label{formula-ii}
J_\phi(z) \kappa_{\phi}(\phi(z), \phi(w))\overline{J_\phi(w)} = 
\frac{1}{2}\Big(\kappa(z, w) - \kappa(\sigma(z), w)\Big), ~ z, w \in U.
\eeq
 \end{enumerate}
 \end{theorem}
\begin{remark} \label{remark-new} 
%Under the assumptions of Theorem~\ref{main-result}, 
We make several remarks for future reference.
\begin{enumerate}
\item[(a)] Since $\phi$ is proper, $J_\phi \neq 0$ (see the first paragraph on \cite[p.~301]{R1980}). 
\item[(b)] Since $\phi$ is $\sigma$-invariant and $J_\sigma=-1,$ by the chain rule, $J_\phi \circ \sigma =- J_\phi,$ and hence $J_\phi$ is nonconstant. 
\item[(c)] Since the zero set of $J_\phi$ has empty interior (by (a)), by the open mapping theorem, $\|\cdot\|_\phi$ defines a norm on the linear space $\mathscr H_\phi(V).$
\item[(d)] Since $\phi$ is $\sigma$-invariant, by (b), $\varGamma_\phi$ maps $\mathscr H_\phi(V)$ into $\mathscr H^\sigma_{-}(U)$. 
\item[(e)] It may happen that $\mathscr H(U)= \mathscr H_\phi(V).$ Indeed, if $U=V=\mathbb D^2,$ $\mathscr H(U)$ is the Hardy space of $\mathbb D^2,$ $\sigma(z_1, z_2)=(-z_1, z_2)$ and $\phi(z_1, z_2)=(z^2_1, z_2),$ then $\kappa(z, w)=4\kappa_{\phi}(z, w)=\frac{1}{(1-z_1 \overline{w}_1)(1-z_2 \overline{w}_2)},$ $z=(z_1, z_2), w=(w_1, w_2) \in \mathbb D^2.$
%Applying Theorem~\ref{main-result} to a weighted Bergman space, one can see that 
\end{enumerate}
\end{remark} 

The formula \eqref{formula-ii} can be seen as a variant of Bell's formula for $2$-proper maps (cf. \cite[Theorem~1]{B1982}, \cite[Corollary~1]{T2013}). 
One of the ingredients of the proof of Theorem~\ref{main-result} is a variant of the removable singularity theorem (see Lemma~\ref{RRST}). Another ingredient in the proof of Theorem~\ref{main-result} is an abstract version of the Waring-Lagrange theorem (see \cite[Theorem~7.1]{AMY2020}) for $\sigma$-invariant holomorphic functions $($see Theorem~\ref{Waring-Lagrange}$).$ For another variant of the Waring-Lagrange theorem, see \cite[Theorem 3.1 \& Subsection~3.1.1]{BDGR2022}.
It is worth noting that a rather special case of Theorem~\ref{Waring-Lagrange} (when $\sigma$ is linear) can be deduced from \cite[Theorem 3.1 \& Subsection~3.1.1]{BDGR2022} and \cite[Proposition~2.2]{G2021}. On the other hand, our proof of Theorem~\ref{Waring-Lagrange} is of analytic flavor and it capitalizes on one of variants of the classical Riemann removable singularity theorem (for short, RRST). Note that the idea of applying $L^2$-version of RRST (see \cite[Theorem~4.2.9]{JP2000}) in obtaining a transformation rule for Bergman kernels appears in \cite{B1982, T2013}. This argument falls short in our case as the spaces under consideration are {\it abstract} functional Hilbert spaces. 
In the first half of Section~\ref{S3}, we apply Theorem~\ref{main-result} to the symmetrized bidisc and the tetrablock to recover known results from the literature (see \cite{BGMR2019, G2021, GR2024, MRZ2013, T2013}). However, as far as we know, the following formula for the {\it Szeg$\ddot{\mbox{o}}$ kernel of the tetrablock $\mathbb E$} appears to be unnoticed. For $z, w$ in a nonempty open subset of $\mathbb E,$
\beq
\label{formula-Tetra}
\kappa_{_{\mathbb E}}(z, w) = 
\sum_{k=0}^{\infty} \binom{3/2}{2k+1} 
\frac{(1-\inp{\sigma(z)}{w}_{_{\mathbb C^3}})^{\frac{3}{2}-2k-1} 4^k (z_1z_2-z_3)^k(\overline{w}_1\overline{w}_2-\overline{w}_3)^{k}}{((1-\inp{\sigma(z)}{w}_{_{\mathbb C^3}})^2-4 (z_1z_2-z_3)(\overline{w}_1\overline{w}_2-\overline{w}_3))^{\frac{3}{2}}},
\eeq
where $\sigma(z_1, z_2, z_3)=(z_1, z_2, -z_3)$ and $\inp{\cdot}{\cdot}_{_{\mathbb C^d}}$ denotes the standard inner-product in $\mathbb C^d.$ Recall that the {\it tetrablock} $\mathbb E,$ first appeared in \cite[Definition~1.1]{AWY2007}, is given by
\beq \label{tetra-new} \mathbb E = \{(z_1, z_2, z_3) \in \mathbb C^3 : 1-z_1z-z_2w+z_3zw \neq 0~\mbox{for all~} z,w \in \overline{\mathbb D}\}.
\eeq
In the second half of Section~\ref{S3}, we apply Theorem~\ref{main-result} 
to construct a Hardy space $H^2(\mathbb H^d)$ over 
the {\it $d$-dimensional fat Hartogs triangle} 
\beq \label{fat-H}
\mathbb H^d:=\{(z_1, \ldots, z_{d-1}, z_d) \in \mathbb C^d : |z_1|^2 < \ldots < |z_{d-1}|^2 < |z_d| <1\}. 
\eeq
The space $H^2(\mathbb H^d)$ is the RKHS associated with the kernel $\kappa_{_{\mathbb H^d}} : \mathbb H^d \times \mathbb H^d \rar \mathbb C$ given by 
\begin{align*}
 \kappa_{_{\mathbb H^d}}(z, w)=\frac{1}{4}\frac{1}{\prod_{j=2}^{d-1}(z_j\overline{w}_j-z_{j-1}\overline{w}_{j-1})}\frac{1}{z_d\overline{w}_d-z_{d-1}^2\overline{w}_{d-1}^2}\frac{1+z_{d-1}\overline{w}_{d-1}}{1-z_d\overline{w}_d}, ~ z, w \in \mathbb H^d,
\end{align*}
where we used the convention that a product over an empty set is $1.$ By Agler's hereditary functional calculus, $\frac{1}{\kappa_{_{\mathbb H^d}}}(T, T^*)$ defines a bounded linear operator for any commuting $d$-tuple $T$ of bounded linear operators on a Hilbert space with its 
Taylor spectrum contained in $\mathbb H^d$
(refer to \cite[Section~4.2]{AMY2020}). 
This together with \cite[Theorem~2.1]{CJP2022} leads to the following variant of von Neumann's inequality.
\begin{corollary} \label{von-N} 
Let $T$ be a commuting $d$-tuple of operators in $\mathcal B(\mathcal H).$ If the Taylor spectrum $\sigma(T)$ of $T$ is contained in the $d$-dimensional fat Hartogs triangle $\mathbb H^d$ and $\frac{1}{\kappa_{_{\mathbb H^d}}}(T, T^*)\Ge 0$ then 
$\|f(T)\| \Le\|f\|_{\infty, \mathbb H^d}$ for every $f \in H^{\infty}(\mathbb H^d).$
\end{corollary}
We also apply Theorem~\ref{main-result} to an egg domain to construct a Hardy space over it (see Subsection~\ref{subs4.2}). For concrete instances of Theorem~\ref{main-result} as presented in Section~\ref{S3}, see Table~\ref{Table1} (see Section~\ref{S3} for relevant notations).

\begin{tiny}
\begin{table}[H]
\begin{tabular}{|c|c|c|c|c|c|}
\hline & & & & & \\ $U$ & $V$ &  $\sigma(z)$ 
 & $\phi(z)$ & $\kappa(z, w)$ &  $\kappa_\phi(z, w)$ \\
 & & & & (modulo constant) & (modulo constant)
\\ 
%\hline  
%& & & & & \\
%$\mathbb D^2$ & $\mathbb D^2$ &  $(-z_1, z_2)$ 
% & $(z^2_1, z_2)$ & $\frac{1}{(1-z_1 \overline{w}_1)(1-z_2 \overline{w}_2)}$ & {\color{blue} $\frac{1}{(1-z_1 \overline{w}_1)(1-z_2 \overline{w}_2)}$}
%\\ 
\hline  
& & & & & \\
$\mathbb D^2$ & $\mathbb G_2$ &  $(z_2, z_1)$ 
 & $(z_1+z_2, z_1z_2)$ & $\frac{1}{(1-z_1 \overline{w}_1)(1-z_2 \overline{w}_2)}$ &  $\frac{1}{(1-z_2\overline{w}_2)^2 - (z_1-z_2\overline{w}_1)(\overline{w}_1-z_1\overline{w}_2)}$ 
\\ 
& & & & & \\
\hline & & & & & \\ $\mathfrak R_{II}$ & $\mathbb E$ &  $(z', -z_3)$ 
 & $(z', z_1z_2-z^2_3)$ & 
  $
\det\left(I-Z\overline{W} \right)^{-\frac{3}{2}},$  
&  See \eqref{formula-Tetra} 
\\ 
  &  &  
 &  &$Z=\big[\begin{smallmatrix} z_{1} & z_3 \\ z_3 & z_{2}\end{smallmatrix}\big]$  &  
\\ 
& & & & & \\
\hline 
& & & & & \\
$\triangle^2$ & $\mathbb H^2$ &  $(z_1, -z_2)$ 
 & $(z_1, z^2_2)$ & $\frac{1}{(1-z_2 \overline{w}_2)(z_{2} \overline{w}_{2}- z_{1} \overline{w}_{1})}$ &  $\frac{1+z_{1}\overline{w}_{1}}{(z_2\overline{w}_2-z_{1}^2\overline{w}_{1}^2)(1-z_2\overline{w}_2)}$ 
\\ 
& & & & & \\
\hline
& & & & & \\
$\mathbb B^2$ & $\mathbb E_{2, 1}$ &  $(z_1, -z_2)$ 
 & $(z_1, z^2_2)$ & $\frac{1}{\big(1-\inp{z}{w}_{_{\mathbb C^2}}\big)^2}$ &  $\frac{1-z_1\overline{w}_1}{[(1-z_1\overline{w}_1)^2-z_2\overline{w}_2]^2}$
\\ 
& & & & & \\
\hline
\end{tabular}
\vskip.2cm
\caption{\label{Table1} Some illustrations of Theorem~\ref{main-result} }
\end{table}
\end{tiny}

 \section{Proof of the Main Theorem}

The proof of Theorem~\ref{main-result} involves several lemmas. We begin with the following known fact (the proof is given for the sake of completeness). 
\begin{lemma} \label{RRST}  For a domain $U$ in $\mathbb C^d,$ 
%let $\mathscr H(U)$ be a functional Hilbert space.   
let $g \in \mathrm{Hol}(U)$ and $f : U \rar \mathbb C$ be a nonzero affine-linear function such that the $Z(f) \subseteq Z(g).$  
%Assume that there exists a sequence $\{g_n\}_{n \Ge 1}$ in $\mathrm{Hol}(U)$ such that 
%\beq \label{assump-RRST}
%\mbox{$\{fg_n\}_{n \Ge 1}$ converges to $g$ in $\mathscr H(U).$} 
%\eeq
Then $\frac{g}{f}$ extends holomorphically to $U.$    
\end{lemma}
\begin{proof}
Assume that $f(z)=\sum_{j=1}^d a_j z_j - b$ for some scalars $a_1, \ldots, a_d, b \in \mathbb C.$ 
If all scalars $a_1, \ldots, a_d$ are zero, then $f$ is constant and the required conclusion holds trivially. 
%If some of scalars $a_1, \ldots, a_d$ are zero, then by considering $f$ as a functions of variables corresponding to those nonzero scalars, we may treat the case in which all  scalars $a_1, \ldots, a_d$ are nonzero. 
Since permutation is a biholomorphism, we may assume that for some $1 \Le l \Le d,$ $a_j \neq 0$ for all $j=1, \ldots, l.$ Consider the new variables 
\beqn
\tilde{z}_1 :=\sum_{j=1}^l a_j z_j, \quad \tilde{z}_j := z_j,~j=2, \ldots, d.
\eeqn 
It follows that $\tilde{z}=Az,$ where $\tilde{z}=(\tilde{z}_1, \ldots, \tilde{z}_d)$ and $z=(z_1, \ldots, z_d)$ are considered as column vectors and $A$ is the $d \times d$ invertible matrix given by  
\beqn
A = 
\left[\begin{matrix}
a_1 & \cdots & a_l & 0 & \cdots & 0 \\
0 & 1 & 0 & 0 & \cdots & 0\\
0 & 0 & 1 & 0 & \cdots & 0\\
\vdots & \vdots & \ddots & \ddots & \vdots & 0\\
0 & 0 & 0 & 0 & \cdots & 1   
\end{matrix}
\right].
\eeqn
In the new coordinates $\tilde{z},$ $U$ gets transformed into $\tilde{U}$ and $\frac{g}{f}$ takes the form $\frac{G}{F},$ where $G(z)=g \circ A^{-1}(z)$ and $F(z)=z_1-b$ for $z \in \tilde{U}.$ Since the zero set of $F$ is contained in the zero set of $G,$ by expanding $G$ around any point in $\{z \in \tilde{U} : z_1=b\}$ as a power series in $z_1$ with coefficients holomorphic in $z_2, \ldots, z_d,$ if required, we can see that $\frac{G}{F}$ extends to a holomorphic function. Since the change of variables is biholomorphic, this completes the proof of the lemma.  
%Also, if $F(w)=0,$ then $Az+b=0$ or $f(z)=0$ and hence by assumption, $g(z)=0.$ Thus $g(A^{-1}w)=0,$ that is, $G(w)=0.$   
\end{proof}

%\begin{proof} 
%Since the reproducing kernel is jointly continuous\footnote{Indeed, since $\eta : (z, w) \mapsto \kappa(z, \overline{w})$ is separately holomorphic (by assumption), by Hartogs separate analyticity theorem, $\eta$ is holomorphic.}, the assumption \eqref{assump-RRST} combined with the reproducing kernel property of $\mathscr H(U)$ (see \eqref{rproperty}) implies that
%\beq \label{cgn-new}
%\mbox{$\{f g_n\}_{n \Ge 1}$ converges compactly to $g.$}
%\eeq 
%%In particular, $g \in \mathrm{Hol}(U).$ 
%Let $h:=\frac{g}{f}$ on $U \backslash Z(f).$ In view of the Hartogs separate analyticity theorem (see \cite{S2005}), it suffices to check that $h$ extends as a separately analytic function on $U.$ To see the latter, first assume that $d > 1.$ Write 
%$a=(a_1, a') \in U,$ where $a_1 \in \mathbb C$ and $a' \in \mathbb C^{d-1}.$ If $w \mapsto f(w, a')$ is nowhere vanishing in a neighborhood, say $W$ of $a_1,$ then $w \mapsto h(w, a')$ is holomorphic in $W.$ If, for some nonnegative integer $k,$ $(\partial^j_{1}f)(a)=0,$ $j=0, \ldots, k,$ then by \eqref{cgn-new} and Weierstrass convergence theorem (see \cite[Theorem~1.4.20]{S2005}) $(\partial^j_{1} g)(a)=0,$  $j=0, \ldots, k.$  It now follows from the one-variable Riemann removable singularity theorem that $w \mapsto h(w, a')$ extends holomorphically in a neighborhood of $a_1.$ Arguing as above (in each variable), one may now show that $h$ extends as a separately analytic function. This completes the proof in case $d > 1.$ It is clear from the discussion above that the same proof works in case $d=1.$  
%\end{proof}

We formally introduce the following notion (weaker than $2$-properness) as it appears naturally in the proof of Theorem~\ref{main-result}. 
\begin{definition} \label{Admissible} Let $U, V$ be domains in $\mathbb C^d.$ 
A holomorphic surjection $\phi : U \rar V$ is said to be {\it admissible} if 
\beq \label{fibre}
\mbox{$\phi^{-1}(\{z\})$ is singleton for every $z \in \phi(Z(J_\phi)).$} 
\eeq
\end{definition}
\begin{remark} \label{2-proper admissible}
Any $2$-proper holomorphic map is admissible. This is a consequence of  \cite[Proposition~15.1.5 \& Theorem~15.1.9]{R1980}. There do exist admissible proper holomorphic maps of higher multiplicities. For example, for a positive integer $m \Ge 2,$ consider the map $\phi : \mathbb C^d \rightarrow \mathbb C^d$ given by 
$$\phi(z) = (z^m_1, z_2, \ldots, z_d), \quad z= (z_1, z_2, \ldots, z_d) \in \mathbb C^d.$$ 
%Note that the Jacobian $J_{\phi}(z)=mz_1^{m-1}.$ Thus, $$Z(J_{\phi})=\{(0, z_2, \ldots, z_d) : z_2, \ldots, z_d \in \mathbb D\}=\phi(Z(J_{\phi})).$$ 
Clearly, $\phi$ is an admissible map. By \cite[Theorem~5.1]{R1982}, $\phi$ is $m$-proper.
\end{remark}
\begin{remark} \label{Fix-sigma} Assume that $\sigma : \mathbb C^d \rar \mathbb C^d$ is affine-linear and $\phi : U \rar V$ is a $\sigma$-invariant admissible map. 
For $w \in Z(J_\phi),$ $\phi(\sigma(w)) = \phi(w) \in \phi(Z(J_\phi)),$ and hence by \eqref{fibre}, $\sigma(w)=w.$ 
Thus we have the inclusion
\beqn
Z(J_\phi) \subseteq \mathrm{Fix}(\sigma).
\eeqn
Also, if $J_\sigma = - 1,$ then by Remark~\ref{remark-new}(b), $J_\phi \circ \sigma = - J_\phi,$ 
 and hence $Z(J_\phi) = \mathrm{Fix}(\sigma).$
\end{remark}

\begin{lemma} \label{Shubham} 
Let
$\phi : U \rar V$ be an admissible proper map and let $h \in \mathrm{Hol}(U)$ be given. Define $f : V \rar \mathbb C$ by $f(z)=h(\tilde{z})$ if $\tilde{z} \in \phi^{-1}(\{z\}).$ If $f$ is well-defined, then $f$ is holomorphic. 
\end{lemma}
\begin{proof} Assume that $f$ is well-defined. 
Since $\phi$ is a proper holomorphic map, by \cite[Remark~15.1.10]{R1980}, $f$ is holomorphic on $V \backslash \phi(Z(J_\phi)).$ 
Note that the proper image of an analytic set is analytic (see \cite[Theorem~4.1.9(c)]{R1980}) and  $\phi(Z(J_\phi))$ is a proper subset of $V$ (since $\phi$ is closed), by \cite[Proposition~4.1.6(3)]{S2005}, the codimension of $\phi(Z(J_\phi))$ is at least $1.$ Thus,  
to see that $f$ is holomorphic on $V,$ by RRST (see \cite[Theorem~4.2.1]{S2005}), it suffices to show that $f$ is continuous on $V$. Clearly, $f$ is continuous on $V \backslash \phi(Z(J_\phi)).$ Now for $z \in \phi(Z(J_\phi)),$ let $\{z_n\}_{n \Ge 1}$ be a sequence in $V$ such that $\{z_n\}_{n \Ge 1}$ converges to $z.$ By \eqref{fibre}, there exists a unique $\tilde z$ such that $\phi(\tilde z)=z.$ For each integer $n \Ge 1,$ let $\tilde {z}_n \in \phi^{-1}(\{z_n\}).$ \vskip.1cm

\noindent
{\bf Claim.} $\{\tilde {z}_n\}_{n \Ge 1}$ converges to $\tilde z.$ \vskip.1cm

%This is equivalent to showing that every subsequence of $\{\tilde {z}_n\}_{n \Ge 1}$  has a further subsequence which converges to $\tilde z.$ 
\noindent
Since $\phi$ is proper, $\phi^{-1}(\{z\}\cup \{z_n: n\Ge 1\})$ is compact. Thus every subsequence of $\{\tilde {z}_n\}_{n \Ge 1}$  has a convergent subsequence. 
%Now we will show that every convergent subsequence of $\{\tilde {z}_n\}_{n \Ge 1}$ converges to $\tilde z.$ 
Let $\{\tilde {z}_{n_k}\}_{k \Ge 1}$ be subsequence of $\{\tilde {z}_n\}_{n \Ge 1}$ converging to some $w \in U.$ By the continuity of  $\phi,$ $\phi(\tilde {z}_{n_k})$ converges to $\phi(w).$ As $\phi(\tilde {z}_{n_k})=z_{n_k}$ for every $k \Ge 1,$ we obtain $\phi(w)=z,$ which implies that $w=\tilde z$ (since, by \eqref{fibre}, $z$ has only one pre-image $\tilde z$). Since any bounded sequence, which is not convergent, has at least two limit points (by Bolzano-Weierstrass theorem), the claim stands verified.
 Since $h$ is a continuous map, by the claim above, $\{f(z_n)\}_{n \Ge 1}$ converges to $f(z).$ 
 %This completes the proof. 
\end{proof}
\begin{remark}
Here is an alternative and short proof of the continuity of $f$ as suggested by the anonymous referee. To see that, the well-definedness of $f$ implies that $h = f \circ \phi.$ Since $h$ is continuous and $\phi$ is a closed continuous map, $f$ is automatically continuous.
\end{remark}

%The completeness of $\mathscr H_\phi(V)$ is a simple consequence of Remark~\ref{2-proper admissible} and the following general fact 
For a special case of the following fact, see \cite[p.~2367]{MRZ2013}. 
\begin{lemma}  \label{fibre-lemma}
Assume that 
$\phi : U \rar V$ is an admissible proper map. Assume further that $J_\phi$ is affine-linear.
Then 
$\mathscr H_\phi(V)$ endowed with the norm $\|\cdot\|_\phi$ is a Hilbert space $($see \eqref{H(V)} and \eqref{phi-norm}$).$  
\end{lemma}
\begin{proof} 
By Remark~\ref{remark-new}(c), $\|\cdot\|_\phi$ defines a norm. To see that $\mathscr H_\phi(V)$ is a Hilbert space, let $\{f_n\}_{n \Ge 1}$ be a Cauchy sequence in $\mathscr H_\phi(V).$ It follows that $\{J_\phi f_n \circ \phi\}_{n \Ge 1}$ is a Cauchy sequence in $\mathscr H(U).$ Since $\mathscr H(U)$ is a Hilbert space, there exists $g \in \mathscr H(U)$ such that 
$\{J_\phi f_n \circ \phi\}_{n \Ge 1}$ converges to $g$ in $\mathscr H(U).$ 
Since $J_\phi$ is nonzero (see Remark~\ref{remark-new}(a)) and affine-linear, by Lemma~\ref{RRST}, $h:=\frac{g}{J_\phi}$ extends as a holomorphic function on $U.$

To complete the proof, it now suffices to show that $h=f \circ \phi$ for some $f \in \mbox{Hol}(V).$ To see this, define $f : V \rar \mathbb C$ by $f(z)=h(\tilde{z})$ if $\tilde{z} \in \phi^{-1}(\{z\}).$ We must check that $f$ is well-defined. For this, note the following:
\begin{enumerate}
\item if $z \notin \phi(Z(J_\phi)),$ then 
\beqn
h(\tilde z) = \frac{g(\tilde z)}{J_\phi(\tilde z)} = \lim_{n \rar \infty} f_n\circ \phi(\tilde{z}) =\lim_{n \rar \infty} f_n(z), 
\eeqn
which is independent of the preimage $\tilde{z},$ 
\item if $z \in \phi(Z(J_\phi)),$ then by the assumption \eqref{fibre}, $\phi^{-1}(\{z\})$ is singleton. 
\end{enumerate}
Thus, in both the cases, $f$ is well-defined. This combined with Lemma~\ref{Shubham} shows that $h=f \circ \phi$ with $f \in \mbox{Hol}(V),$ which completes the proof. 
\end{proof}

%\section{A variant of Waring–Lagrange theorem for holomorphic functions}
The proof of Theorem~\ref{main-result} relies on the following variant of the Waring-Lagrange theorem (cf. \cite[Theorem~7.1]{AMY2020}). 
\begin{theorem} \label{Waring-Lagrange}
Let $\sigma : \mathbb C^d \rar \mathbb C^d$ be an affine-linear involution such that $J_\sigma = - 1$ and 
%If there exists a $\sigma$-invariant $2$-proper holomorphic map $\phi : U \rightarrow V$ such that $J_\phi \circ \sigma = - J_\phi,$ 
let $\phi : U \rar V$ be a $\sigma$-invariant $2$-proper holomorphic map.
Then $\{g \in \mathrm{Hol}(U): g \circ \sigma = g \} = \{ f \circ \phi: f \in \mathrm{Hol}(V)\}.$
\end{theorem}
\begin{proof}
%[Proof of Theorem~\ref{Waring-Lagrange}] 
%Let $f \in \mathrm{Hol}(V).$ 
Since $\phi$ is $\sigma$-invariant, we have $$\{ f \circ \phi: f \in \mathrm{Hol}(V)\} \subseteq \{g\in \mathrm{Hol}(U): g \circ \sigma = g \}.$$ To see the reverse inclusion, let $g \in \mathrm{Hol}(U)$ such that $g \circ \sigma = g.$ Define a function $f : V \rightarrow \mathbb C$ by setting $f(z)= g(\tilde z),$ where $z \in V$ and $\tilde z \in \phi^{-1}(\{z\}).$ \vskip.1cm

\noindent
{\bf Claim.} $f$ is well-defined. \vskip.1cm

\noindent
By Remark~\ref{2-proper admissible}, $\phi$ is admissible (see Definition~\ref{Admissible}). Consequently, by \eqref{fibre}, $f$ is well-defined on $\phi(Z(J_\phi)).$ If $\sigma(\tilde z) = \tilde z,$ then by Remark~\ref{Fix-sigma} and the assumption that $J_\sigma = - 1,$ $\tilde z \in Z(J_\phi),$ that is, $J_\phi(\tilde z)=0$ or $z \in \phi(Z(J_\phi))$. To complete the verification of the claim, let $z \notin \phi(Z(J_\phi)),$ and note that $\sigma(\tilde z) \neq \tilde z.$ Since $\phi$ is $2$-proper,  $\phi^{-1}(\{z\})=\{ \tilde z, \sigma (\tilde z)\}.$ Now the $\sigma$-invariance of $g$ yields the well-definedness of $f$ on $V \backslash \phi(Z(J_\phi)).$ 
One may now apply Lemma~\ref{Shubham} to complete the proof.  
\end{proof} 
\begin{proof}[Proof of Theorem~\ref{main-result}] 
 The completeness of $\mathscr H_\phi(V)$ is a simple consequence of Remark~\ref{2-proper admissible} and Lemma~\ref{fibre-lemma}. 
 
(i) To see this, let $g \in \mathscr H^{\sigma}_{-}(U).$  
Since $\varGamma_\phi$ maps $\mathscr H_\phi(V)$ into $\mathscr H^\sigma_{-}(U)$ (see Remark~\ref{remark-new}(d)), it suffices to check that there exists $f \in \mathscr H_\phi(V)$ such that $\varGamma_\phi(f)=g.$ We claim that $\frac{g}{J_\phi}$ extends as a holomorphic function to $U.$
Since $J_\sigma =-1,$ by Remark~\ref{Fix-sigma}, $Z(J_\phi)=\mathrm{Fix}(\sigma).$ On the other hand, $g \circ \sigma = -g.$ Thus, for any $w \in Z(J_\phi),$ $g(w)= g \circ \sigma(w) = -g(w).$ This yields the inclusion 
\beq \label{inclusion-new}
Z(J_\phi) \subseteq Z(g).
\eeq
%However, by assumption, $\phi_1, \ldots, \phi_{d-1}$ are homogeneous polynomials of degree $1$ and $\phi_d$ is a homogeneous polynomial of degree $2$ such that $\phi^{-1}(\{0\})=\{0\}.$ A simple calculation shows that $J_\phi$ is a polynomial in $z$ of degree at most $1.$ 
By assumption, $J_\phi$ is affine-linear. 
This combined with \eqref{inclusion-new} and Lemma~\ref{RRST} shows that $\frac{g}{J_\phi}$ extends as a holomorphic function, say $h \in \mathrm{Hol}(U).$
As $h$ and $\frac{g}{J_\phi}$ are equal and $\sigma$-invariant outside the nowhere dense set $Z(J_\phi),$ $h$ is also $\sigma$-invariant.  This combined with the assumptions that $\phi$ is a $2$-proper map and $J_\sigma = -1$ shows that there exists $f \in \mathrm{Hol}(V)$ such that $ h= f \circ \phi$ (see Theorem~\ref{Waring-Lagrange}). It follows that $\varGamma_\phi(f)=g.$

(ii) We now show that $\mathscr H_\phi(V)$ is a reproducing kernel Hilbert space (cf. \cite[Lemma 3.3]{GR2024})). 
To see that, consider the complex vector space $H = \{f \in \mbox{Hol}(U) : f = -f \circ \sigma\}.$ Since $H$ is closed in $\mbox{Hol}(U)$ (for the topology of compact convergence), it is a Fr\'{e}chet space. 
By Lemma~\ref{RRST} and Theorem~\ref{Waring-Lagrange}, there is a linear mapping $T : H \rar \mbox{Hol}(V)$ such that $f = J_\phi((Tf) \circ \phi) = \Gamma_\phi(Tf)$ for every $f \in H$ (indeed, $T$ is the set-theoretic inverse of $\Gamma_\phi : \mbox{Hol}(V) \rar H$). 
By the bounded inverse theorem, $T$ is continuous. Since the canonical mapping $\mathscr H_\phi(V) \hookrightarrow \mbox{Hol}(V)$ is composition of the isometry $\Gamma_\phi :  \mathscr H_\phi(V) \rar \mathscr H^\sigma_-(U),$ the canonical (continuous) mapping $\mathscr H^\sigma_-(U) \hookrightarrow H,$ and $T,$ it is continuous. Hence, $\mathscr H_\phi(V)$ is a reproducing kernel Hilbert space.

By (i), note that for any $w \in U,$ there exists $f_w \in \mathscr H_\phi(V)$ such that $\kappa^{\sigma}_{-}(\cdot, w)=J_\phi \cdot f_w \circ \phi.$ 
 Thus, for each $w \in U,$ $\frac{\kappa^{\sigma}_{-}(\cdot, w)}{J_\phi}$ extends to the holomorphic function $f_w \circ \phi$ on $U.$ 
Further, note that for each $z \in U,$ the function $F_z : w \mapsto \overline{f_w \circ \phi(z)}$ belongs to $\mathscr H^\sigma_-(U).$ 
Define
 \beqn
\kappa_\phi(\phi(z), \phi(w)) = 
 \inp{F_w}{F_z}_{\mathscr H^\sigma_-(U)}, \quad z, w \in U. \eeqn
It is easy to see using the reproducing property of $\mathscr H^\sigma_-(U)$ (see \eqref{rproperty}) that 
\beqn
J_\phi(z) \kappa_{\phi}(\phi(z), \phi(w))\overline{J_\phi(w)} &=& \kappa^{\sigma}_{-}(z, w) \\
&\overset{\eqref{kappa-s-minus}}= &
\frac{1}{2}\Big(\kappa(z, w) - \kappa(\sigma(z), w)\Big), \, z, w \in U
\eeqn
(if $w \in Z(J_\phi)$, then by Remark~\ref{Fix-sigma}, $w \in \mathrm{Fix}(\sigma),$ and hence by \eqref{sigma-invariant}, $\kappa(z, w)=\kappa(z, \sigma(w))=\kappa(\sigma(z), w)$, $z, w \in U$).

To complete the proof, it suffices to show that $\kappa_\phi$ is the reproducing kernel for $\mathscr H_\phi(V).$ 
To see that, 
let $f \in \mathscr H_\phi(V)$ and $w \in U.$  We first assume that $J_\phi(w) \neq 0$. Since $\varGamma_\phi$ is unitary from $\mathscr H_\phi(V)$ onto $\mathscr H^{\sigma}_{-}(U)$ (by (i)), 
\beq \label{rkp-dense} \notag
\inp{f}{\kappa_\phi(\cdot, \phi(w))}_{\mathscr H_\phi(V)}  &=& \inp{\varGamma_\phi \,f}{\varGamma_\phi \,\kappa_\phi(\cdot, \phi(w))}_{\mathscr H^\sigma_-(U)}\\ \notag &=& \inp{J_\phi \cdot f\circ \phi}{J_\phi \cdot \kappa_\phi(\phi(\cdot), \phi(w))}_{\mathscr H^\sigma_-(U)} \\ \notag
&\overset{\eqref{formula-ii}}=& \frac{1}{J_\phi(w)} \inp{J_\phi \cdot f\circ \phi}{\kappa^{\sigma}_{-}(\cdot, w)}_{\mathscr H^\sigma_-(U)} \\
&=& f(\phi(w)).
\eeq
This shows that $\mathscr H_\phi(V)$ has reproducing kernel property at all points in $V \backslash \phi(Z(J_\phi)).$ Since $\mathscr H_\phi(V)$ is a reproducing kernel Hilbert space, a density argument using \eqref{rkp-dense} now completes the proof. 
\end{proof}
\begin{remark}
Note that Theorem~\ref{main-result} is applicable to $\phi=(\phi_1, \ldots, \phi_d)$ provided $\phi_1, \ldots, \phi_{d-1}$ are homogeneous polynomials of degree $1$ and $\phi_d$ is a homogeneous polynomial of degree $2$ such that $\phi^{-1}(\{0\})=\{0\}$. Indeed, by \cite[Theorem~5.1]{R1982}, such a map $\phi$ is $2$-proper.
\end{remark}

%We conclude this section with a connection between Theorem~\ref{main-result} and boundary behavior of $2$-proper maps. 
Let $\sigma, U, V, \phi$ be as in the statement of Theorem~\ref{main-result}. Assume that 
$\phi_1, \ldots, \phi_d$ are multipliers of $\mathscr H(U).$ Let $f\in \mathscr H_\phi(V).$ Then $J_\phi \cdot f\circ \phi \in \mathscr H(U),$ and hence by our assumption $\phi_i \cdot J_\phi \cdot f\circ \phi \in \mathscr H(U),$ $i=1, \ldots, d.$  It follows that 
\beqn
%\|z_i f\|_{\phi}=
\|\phi_i \cdot J_\phi \cdot f\circ \phi \|_{\mathscr H(U)}&=&\|J_\phi \cdot \phi_i \cdot f\circ \phi\|_{\mathscr H(U)}\\
&=&\|J_\phi \cdot (z_if)\circ \phi\|_{\mathscr H(U)}, \quad i=1, \ldots, d.
%=\| \phi_i \varGamma_\phi f\|_{\mathcal H(U)}
\eeqn
Thus, for $i=1,2,\ldots,d,$ $z_i f \in \mathscr H_\phi(V)$ and for each $z\in U,$
\beqn
(\varGamma_\phi \mathscr M_{z_i}f)(z)=J_\phi(z) \phi_i(z) f(\phi(z))
=(\phi_i \cdot \varGamma_\phi f)(z)=(\mathscr M_{\phi_i} \varGamma_\phi f)(z),
\eeqn
where $\mathscr M_\psi$ denotes the operator of multiplication by $\psi.$  
This combined with Theorem~\ref{main-result} yields the following: 
\begin{corollary} Let $\sigma, U, V, \phi$ be as in the statement of Theorem~\ref{main-result}. Assume that $U, V$ are bounded.
%Under the assumptions of Theorem~\ref{main-result}, 
If  $\phi_1, \ldots, \phi_d$ are multipliers of $\mathscr H(U),$ then 
$\phi_1, \ldots, \phi_d$ are multipliers of $\mathscr H^\sigma_{-}(U).$ In this case,  
$z_1, \ldots, z_d$ are multipliers of $\mathscr H_\phi(V)$ and 
the commuting $d$-tuple $\mathscr M_\phi = (\mathscr M_{\phi_1}, \ldots, \mathscr M_{\phi_d})$ on 
$\mathscr H^\sigma_{-}(U)$ is unitarily equivalent to the commuting $d$-tuple $\mathscr M_z=(\mathscr M_{z_1}, \ldots, \mathscr M_{z_d})$ on 
$\mathscr H_\phi(V).$ 
\end{corollary}

\section{Illustrations} \label{S3}
In this section, we first apply Theorem~\ref{main-result} to obtain Hardy spaces of the symmetrized bidisc and tetrablock. These constructions recover \cite[Theorem 3.1]{MRZ2013} (in two variables) and \cite[Example 3.11]{GR2024}. We also apply Theorem~\ref{main-result} to construct Hardy spaces of the $d$-dimensional fat Hartogs triangle and an egg domain. 

\subsection{Hardy space of the symmetrized bidisc}

%\begin{example}[Hardy space of the symmetrized bidisc] 
Consider the {\it symmetrized bidisc} given by 
\beqn
\mathbb G_2 = \{(z_1+z_2, z_1z_2) \in \mathbb C^2 : (z_1, z_2) \in \mathbb D^2\}.
\eeqn
This domain first appeared in \cite{AY1999}. 
Let $U = \mathbb D^2$ and $V = \mathbb G_2.$ Let $\sigma : \mathbb C^2 \rar \mathbb C^2$ be the involution given by $\sigma(z_1, z_2)=(z_2, z_1),$ $(z_1, z_2) \in \mathbb C^2.$ 
Clearly, $U$ is $\sigma$-invariant and $J_\sigma =-1.$ Consider the map  $\phi$ given by 
\beqn
\phi(z)= (z_1+z_2, z_1z_2), \quad z=(z_1, z_2).
\eeqn
By \cite[Theorem~5.1]{R1982}, $\phi$ is $2$-proper. Clearly, $\phi$ is $\sigma$-invariant and it maps $U$ onto $V$.  
Let $\mathscr H(U)$ be the {\it Hardy space $H^2(\mathbb D^2)$ of the unit bidisc $\mathbb D^2,$} that is, the reproducing kernel Hilbert space with the reproducing kernel $$\kappa(z, w)=\frac{1}{\prod_{j=1}^2 (1-z_j \overline{w}_j)}, \quad z=(z_1, z_2), w=(w_1, w_2) \in \mathbb D^2.$$ By Theorem~\ref{main-result}, we obtain the reproducing kernel Hilbert space $\mathscr H_\phi(V).$
Moreover, $\varGamma_\phi$ maps $\mathscr H_\phi(V)$ isometrically onto $$\mathscr H^{\sigma}_-(U)=\{f \in H^2(\mathbb D^2) : f(z_1, z_2)=-f(z_2, z_1)\}.$$  On the other hand, for $z, w \in U,$ 
\beqn
\kappa^\sigma_{-}(z, w)&\overset{\eqref{kappa-s-minus}}=&\frac{1}{2}\Big(\frac{1}{(1-z_1\overline{w}_1)(1-z_2\overline{w}_2)} - \frac{1}{(1-z_2\overline{w}_1)(1-z_1\overline{w}_2)}\Big) \\
%&=& \frac{1}{2}\Big(\frac{(1-z_2\overline{w}_1)(1-z_1\overline{w}_2) - (1-z_1\overline{w}_1)(1-z_2\overline{w}_2)}{(1-z_1\overline{w}_1)(1-z_2\overline{w}_2)(1-z_2\overline{w}_1)(1-z_1\overline{w}_2)}\Big) \\
&=& \frac{1}{2}\frac{(z_1-z_2)(\overline{w}_1-\overline{w}_2)}{(1-z_1\overline{w}_1)(1-z_2\overline{w}_2)(1-z_2\overline{w}_1)(1-z_1\overline{w}_2)}.
\eeqn  
Since $J_\phi(z)=z_1-z_2,$ 
it now follows from \eqref{formula-ii} that the reproducing kernel $\kappa_\phi$ of $\mathscr H_\phi(V)$ is given by
$$
\kappa_\phi(\phi(z), \phi(w)) = \frac{1}{2}\frac{1}{(1-z_1\overline{w}_1)(1-z_2\overline{w}_2)(1-z_2\overline{w}_1)(1-z_1\overline{w}_2)}, ~z, w \in U.
$$
This recovers \cite[Theorem 3.1]{MRZ2013} in two variables. It can be seen that $\kappa_\phi(z, w) = \frac{1}{2}\frac{1}{(1-z_2\overline{w}_2)^2 - (z_1-z_2\overline{w}_1)(\overline{w}_1-z_1\overline{w}_2)},$ $z, w \in V$ (see \cite[p.~513]{BS2018}, \cite[Corollary~2.4]{MRZ2013}).
%\end{example}

\subsection{Hardy space of the tetrablock} 
Following \cite[p.~5]{H1979}, the {\it Cartan domain of type II} in $\mathbb C^3$ is given by
$$\mathfrak R_{II} = \{(z_{1}, z_{2}, z_3) \in \mathbb C^3 : \big\|\big[\begin{smallmatrix} z_{1} & z_3 \\ z_3 & z_{2}\end{smallmatrix}\big]\big\|< 1\}.$$ 
By \cite[Eqs~(4.5.1)\&(4.7.6)]{H1979}, the kernel 
\beqn
\kappa(z, w) = \frac{1}{\mathrm{Vol}(\mathfrak R_{II})}
\det\left(I-\big[\begin{smallmatrix} z_{1} & z_3 \\ z_3 & z_{2}\end{smallmatrix}\big] \big[\begin{smallmatrix} \bar{w}_1 & \bar{w}_3 \\
\bar{w}_3 & \bar{w}_2\end{smallmatrix}\big]\right)^{-\frac{3}{2}}, \quad z, w \in \mathfrak R_{II} 
\eeqn
is positive semi-definite. The {\it Hardy space $H^2(\mathfrak R_{II})$ of the Cartan domain $\mathfrak R_{II}$ of type II} is the reproducing kernel Hilbert space associated with $\kappa$ (the equivalence of this definition with the one given on \cite[p.~521]{HM1969} follows from \cite[Theorem~4.6.1]{H1979}).
%The {\it tetrablock}, first appeared in \cite[Definition~1.1]{AWY2007}, is given by
%$$\mathbb E = \{(z_1, z_2, z_3) \in \mathbb C^3 : 1-z_1z-z_2w+z_3zw \neq 0~\mbox{for all~} z,w \in \overline{\mathbb D}\}.$$

Let $U=\mathfrak R_{II}$ and $V=\mathbb E$ (see \eqref{tetra-new}). 
Let $\sigma : \mathbb C^3 \rar \mathbb C^3$ be the involution given by $$\sigma(z_1, z_2, z_3)=(z_1, z_2, -z_3), \quad (z_1, z_2, z_3) \in \mathbb C^3.$$ 
Since $\big[\begin{smallmatrix} 1 & 0 \\ 0 & -1\end{smallmatrix}\big]\big[\begin{smallmatrix} z_{1} & z_3 \\ z_3 & z_{2}\end{smallmatrix}\big]\big[\begin{smallmatrix} 1 & 0 \\ 0 & -1\end{smallmatrix}\big]=\big[\begin{smallmatrix} z_{1} & -z_3 \\ -z_3 & z_{2}\end{smallmatrix}\big],$ $U$ is $\sigma$-invariant. Clearly, $J_\sigma =-1.$
Consider the map $\phi= (\phi_1, \phi_2, \phi_3)$ given by 
\beqn
\phi_1(z)=z_1, \phi_2(z)=z_2, \phi_3(z)=z_1z_2-z^2_3, \quad z=(z_1, z_2, z_3).
\eeqn
By \cite[Theorem 2.2]{AWY2007}, $\phi$ 
maps $U$ onto $V.$ Moreover, $\phi$ is a $2$-proper map (see \cite[Theorem~5.1]{R1982}) and $\sigma$-invariant. 
Applying Theorem~\ref{main-result} to $\mathscr H(U):=H^2(\mathfrak R_{II})$ with above choices of $\phi$ and $\sigma,$ we obtain the reproducing kernel Hilbert space $\mathscr H_\phi(V).$ Moreover, $\varGamma_\phi$ maps $\mathscr H_\phi(V)$ isometrically onto $\mathscr H^{\sigma}_-(U).$
 
%By Proposition~\ref{kernel-plus-minus}, for $z, w \in U,$ 
%\beqn
%\kappa^\sigma_{-}(z, w)&=&\frac{1}{2}\Big(\Big) \\
%&=& .
%\eeqn  

In the following discussion, we describe the reproducing kernel of $\mathscr H_\phi(V).$
A routine calculation using
$$I-\begin{bmatrix} z_{1} & z_3 \\ z_3 & z_{2}\end{bmatrix} \begin{bmatrix} \bar{w}_1 & \bar{w}_3 \\
\bar{w}_3 & \bar{w}_2\end{bmatrix} = \begin{bmatrix}
1 - z_1 \bar{w}_1 - z_3 \bar{w}_3 & -z_1 \bar{w}_3 - z_3 \bar{w}_2 \\
-z_3 \bar{w}_1 - z_2 \bar{w}_3 & 1 - z_3 \bar{w}_3 - z_2 \bar{w}_2
\end{bmatrix}
$$
shows that 
\begin{align*}
& \det(I-\big[\begin{smallmatrix} z_{1} & z_3 \\ z_3 & z_{2}\end{smallmatrix}\big] \big[\begin{smallmatrix} \bar{w}_1 & \bar{w}_3 \\
\bar{w}_3 & \bar{w}_2\end{smallmatrix}\big]) 
%&= \left( 1 - z_1 \bar{w}_1 - z_3 \bar{w}_3 \right) \left( 1 - z_3 \bar{w}_3 - z_2 \bar{w}_2 \right) - \left( -z_1 \bar{w}_3 - z_3 \bar{w}_2 \right) \left( -z_3 \bar{w}_1 - z_2 \bar{w}_3 \right) \\
%&= \left( 1 - z_1 \bar{w}_1 - z_3 \bar{w}_3 \right) \left( 1 - z_3 \bar{w}_3 - z_2 \bar{w}_2 \right) - \left( z_1 \bar{w}_3 + z_3 \bar{w}_2 \right) \left( z_3 \bar{w}_1 + z_2 \bar{w}_3 \right)\\ &=  1 - z_1 \bar{w}_1 - z_2 \bar{w}_2 - 2z_3 \bar{w}_3 + z_1 z_2 \bar{w}_1 \bar{w}_2 + z_1 z_3 \bar{w}_1 \bar{w}_3 \\ &+  z_2 z_3 \bar{w}_2 \bar{w}_3 + z_3^2 \bar{w}_3^2 - z_1 z_3 \bar{w}_1\bar{w}_3 - z_1 z_2 \bar{w}_3^2 - z_3^2 \bar{w}_2\bar{w}_1 - z_2z_3 \bar{w}_2\bar{w}_3 \\ &=  1 - z_1 \bar{w}_1 - z_2 \bar{w}_2 - 2z_3 \bar{w}_3 + z_1 z_2 \bar{w}_1 \bar{w}_2 + z_3^2 \bar{w}_3^2 - z_1 z_2 \bar{w}_3^2 - z_3^2 \bar{w}_2\bar{w}_1 \\ &= 
\\
&=1 - z_1 \bar{w}_1 - z_2 \bar{w}_2 - 2z_3 \bar{w}_3+(z_1z_2-z_3^2)(\bar{w}_1 \bar{w}_2-\bar{w}_3^2), 
\end{align*}
and hence, we have
$$\kappa(z, w) = \frac{1}{\mathrm{Vol}(\mathfrak R_{II})}  \Big[1 - \phi_1(z) \overline{\phi_1(w)} - \phi_2(z) \overline{\phi_2(w)} - 2z_3 \bar{w}_3+\phi_3(z) \overline{\phi_3(w)}\Big]^{-\frac{3}{2}}.
%(1 - z_1 \bar{w}_1 - z_2 \bar{w}_2 - 2z_3 \bar{w}_3+(z_1z_2-z_3^2)(\bar{w}_1 \bar{w}_2-\bar{w}_3^2))^{-\frac{3}{2}}.
$$
%It
%$$[\det(I-\sigma(z)\overline{W})]^{\frac{-3}{2}} = (1 - z_1 \bar{w}_1 - z_2 \bar{w}_2 + 2z_3 \bar{w}_3+(z_1z_2-z_3^2)(\bar{w}_1 \bar{w}_2-\bar{w}_3^2))^{\frac{-3}{2}} $$
Since $J_\phi(z)= -2z_3,$ it now follows from \eqref{formula-ii} that 
\begin{align*}
  &(2z_3)\kappa_{\phi}(\phi(z), \phi(w))(2\bar{w}_3) \\
%&= \frac{1}{2\mathrm{Vol}(\mathfrak R_{II})}\Big([1 - z_1 \bar{w}_1 - z_2 \bar{w}_2 - 2z_3 \bar{w}_3+(z_1z_2-z_3^2)(\bar{w}_1 \bar{w}_2-\bar{w}_3^2)]^{-\frac{3}{2}} \\ &- [1 - z_1 \bar{w}_1 - z_2 \bar{w}_2 + 2z_3 \bar{w}_3+(z_1z_2-z_3^2)(\bar{w}_1 \bar{w}_2-\bar{w}_3^2)]^{-\frac{3}{2}}\Big)\\
&= \frac{1}{2\mathrm{Vol}(\mathfrak R_{II})}\Big(\Big[1 - \phi_1(z) \overline{\phi_1(w)} - \phi_2(z) \overline{\phi_2(w)} - 2z_3 \bar{w}_3+\phi_3(z) \overline{\phi_3(w)}\Big]^{-\frac{3}{2}} \\ &- \Big[1 - \phi_1(z) \overline{\phi_1(w)} - \phi_2(z) \overline{\phi_2(w)} + 2z_3 \bar{w}_3+\phi_3(z) \overline{\phi_3(w)}\Big]^{-\frac{3}{2}}\Big).
\end{align*} 
To find a neat expression for $\kappa_\phi,$ consider the open set 
\beqn
\Omega = \{z \in \mathbb C^3 : 2|z_3|^2 < 1-|z_1|^2-|z_2|^2+|z_1z_2-z^2_3|^2\},
\eeqn 
and note that $\Omega \cap \mathfrak{R}_{II} \neq \emptyset$ (for example, consider $(1/2, 0, 0)$ or $(0, 0, 1/2)$). We claim that $\kappa_\phi(\phi(z), \phi(w))$ can be expressed solely in terms $\phi_1, \phi_2, \phi_3$ over $\Omega \cap \mathfrak{R}_{II}.$ To see this, let 
\beqn
A(z) = 1 - |\phi_1(z)|^2  - |\phi_2(z)|^2 + |\phi_3(z)|^2, ~ B(z) =  2|z_3|^2, \quad z \in \Omega.
\eeqn
Note that for $z \in \Omega \setminus Z(B),$ 
\beqn
&& \frac{1}{B(z)}\Big(\frac{1}{(A(z)-B(z))^{\frac{3}{2}}}-\frac{1}{(A(z)+B(z))^{\frac{3}{2}}}\Big) \\
&=& \frac{1}{B(z)}\frac{(A(z)+B(z))^{\frac{3}{2}} - (A(z)-B(z))^{\frac{3}{2}}}{(A(z)^2-B(z)^2)^{\frac{3}{2}}} \\
&=& 2\sum_{k=0}^{\infty} \binom{3/2}{2k+1} \frac{A(z)^{\frac{3}{2}-2k-1}B(z)^{2k}}{(A(z)^2-B(z)^2)^{\frac{3}{2}}}.
\eeqn
This is a function in $A(z)$ and $B(z)^2,$ which indeed are functions in $\phi_1, \phi_2, \phi_3.$ Hence, after polarizing the real analytic function of $\phi(z),$ we obtain a function holomorphic in $\phi(z)$ and anti-holomorphic in $\phi(w).$ This yields the formula \eqref{formula-Tetra} for $\kappa_\phi(z, w)=\kappa_{\mathbb E}(z, w)$ (modulo constant).

\subsection{Hardy space for the $d$-dimensional fat Hartogs triangle}
For a positive integer $d \Ge 2,$ Let $\triangle^d =\{(z_1, \ldots, z_d) \in \mathbb C^d : |z_1| < \ldots < |z_d| <1\}$ be the {\it $d$-dimensional Hartogs triangle} in $\mathbb C^d.$ 
Recall from \cite{CJP2023, M2021} that the {\it Hardy space $H^2(\triangle^d)$ of $\triangle^d$} is the reproducing kernel Hilbert space associated with the reproducing kernel given by 
\beqn
\kappa(z, w) = \frac{1}{1-z_d \overline{w}_d}  \prod_{j=2}^{d} \frac{1}{z_{j} \overline{w}_{j}- z_{j-1} \overline{w}_{j-1}}, \quad z, w \in \triangle_{d}.  
\eeqn
%Consider the fat Hartogs triangle $\mathbb H^d$  given by \eqref{fat-H}.
%\beqn
%\mathbb H^d = \{(z_1, \ldots, z_d) \in \mathbb C^d : |z_1|^2 < \ldots < |z_{d-1}|^2 < |z_d| <1\}.
%\eeqn
Let $U=\triangle^d$ and $V=\mathbb H^d$ (see \eqref{fat-H}). Consider the involution $\sigma : \mathbb C^d \rar \mathbb C^d$ given by 
\beq \label{sima-new}
\sigma(z)=(z', -z_d), \quad z=(z', z_d) \in \mathbb C^{d-1} \times \mathbb C = \mathbb C^d.
\eeq
Note that $J_\sigma = -1.$
Consider the map $\phi=(\phi_1, \ldots, \phi_d)$ given by 
\beq \label{phi-new}
\phi_j(z)=z_j, ~j=1, \ldots, d-1, ~ \phi_d(z)=z^2_d, \quad z=(z_1, \ldots, z_d).
\eeq
By \cite[Theorem~5.1]{R1982}, $\phi$ is $2$-proper. Clearly, $\phi$ maps $U$ onto $V$ and it is $\sigma$-invariant. Applying Theorem~\ref{main-result} to $\mathscr H(U):=H^2(U)$ with above choices of $\phi$ and $\sigma,$ we obtain the reproducing kernel Hilbert space $\mathscr H_\phi(V).$ Moreover, $\varGamma_\phi$ maps $\mathscr H_\phi(V)$ isometrically onto $\mathscr H^{\sigma}_-(U).$ Let us compute the reproducing kernel of 
$\mathscr H_\phi(V).$ 
Set $\eta(z, w)= \frac{1}{2}\prod_{j=2}^{d-1}\frac{1}{z_j\overline{w}_j-z_{j-1}\overline{w}_{j-1}},$ $z, w \in U.$ Since $\eta(\sigma(z), w)=\eta(z, w),$   
\begin{align*}
   \kappa_-^{\sigma}(z, w) &=\frac{\eta(z, w)}{(z_d\overline{w}_d-z_{d-1}\overline{w}_{d-1})(1-z_d\overline{w}_d)}+\frac{\eta(z, w)}{(z_d\overline{w}_d+z_{d-1}\overline{w}_{d-1})(1+z_d\overline{w}_d)}  \\&=  \eta(z, w)\,\frac{2(1+z_{d-1}\overline{w}_{d-1})z_d\overline{w}_d}{(z_d^2\overline{w}_d^2-z_{d-1}^2\overline{w}_{d-1}^2)(1-z_d^2\overline{w}_d^2)}, \quad z, w \in U. 
\end{align*}
Since $J_{\phi}(z)=2z_d,$ by \eqref{formula-ii}, for $z, w \in U,$
\begin{align*}
 (2z_d)\kappa_{\phi}(\phi(z), \phi(w)) (2\overline{w}_d) = \frac{2\eta(z, w)(1+z_{d-1}\overline{w}_{d-1})z_d\overline{w}_d}{(z_d^2\overline{w}_d^2-z_{d-1}^2\overline{w}_{d-1}^2)(1-z_d^2\overline{w}_d^2)}.
 \end{align*}
%\begin{align*}
% 4\kappa_{\phi}(\phi(z), \phi(w)) = \prod_{j=2}^{d-1}\frac{1}{\phi_j(z)\overline{\phi_j(w)}-\phi_{j-1}(z)\overline{\phi_{j-1}(w)}}\frac{(1+\phi_{d-1}(z)\overline{\phi_{d-1}(w)}}{(\phi_d(z)\overline{\phi_d(w)}-\phi_{d-1}(z)^2\overline{\phi_{d-1}(w)}^2)(1-\phi_d(z)\overline{\phi_d(w))}}
%\end{align*}
It is now easy to see that 
\begin{align*}
 \kappa_{\phi}(z, w) = \frac{1}{2}\frac{\eta(z, w)(1+z_{d-1}\overline{w}_{d-1})}{(z_d\overline{w}_d-z_{d-1}^2\overline{w}_{d-1}^2)(1-z_d\overline{w}_d)}, \quad z, w \in V.
\end{align*}
We refer to $\mathscr H_\phi(V)$ as the {\it Hardy space $H^2(\mathbb H^d)$} of the $d$-dimensional fat Hartogs triangle $\mathbb H^d.$ 
This terminology may be justified in view of the following proposition.
\begin{proposition} \label{Hardy-int}
For $f \in H^2(\mathbb H^d),$ $\|f\|^2_{\phi}$ is given by 
\begin{equation*}
\sup_{\underset{t_j \in (0, 1)}{j=1, \ldots, d}} \int_{[0, 2\pi]^d}  \!\Big|f\Big(\prod_{j=1}^d t_j e^{i \theta_1}, \prod_{j=2}^d t_j e^{i \theta_2}, \ldots, t_d^2 e^{2i \theta_d}\Big)\Big|^2 t^{2d+1}_d\prod_{j=1}^{d-1} t^{2j-1}_j \frac{d\theta}{2^{d-2}\pi^d}. 
\end{equation*}
\end{proposition}
\begin{proof} 
This is a direct application of \eqref{phi-norm} and the integral formula \cite[Proposition~7.5(i)]{CJP2023} for the norm on $H^2(\triangle^d).$
\end{proof}
The space $H^2(\mathbb H^2)$ has been constructed in \cite[Theorem~6.2]{GGLV2021} (the case in which $m=2,$ $n=1$ and $k=1$).

\begin{proof}[Proof of Corollary~\ref{von-N}]
It is easy to see from Proposition~\ref{Hardy-int} and \cite[Corollary 5.22]{PR2016} that the multiplier algebra of $H^2(\mathbb H^d)$ is $H^{\infty}(\mathbb H^d)$ with equality of norms. This combined with \cite[Theorem~2.1]{CJP2022} completes the proof. 
\end{proof}

\subsection{Hardy space of an egg domain} \label{subs4.2}
For an integer $d \geqslant 2,$ let $\mathbb B^d$ denote the open unit ball in $\mathbb C^d.$ 
Recall that the {\it Hardy space $H^2(\mathbb B^d)$ of $\mathbb B^d$} is the reproducing kernel Hilbert space associated with the {\it Szeg$\ddot{\mbox{o}}$ kernel} $\kappa(z, w)=\big(1-\inp{z}{w}_{_{\mathbb C^d}}\big)^{-d},$ $z, w \in \mathbb B^d.$ We are interested in the {\it egg domain} $$\mathbb E_{d, 1}=\{(z_1, \ldots, z_{d-1}, z_d) \in \mathbb C^d : |z_1|^2 + \cdots + |z_{d-1}|^2 + |z_d| < 1\}$$ (the Bergman kernel of this domain appears in \cite[Section~3.2]{BFS1999}; refer to \cite{CS1985} for operator theory on general egg domains).  

Let $U=\mathbb B^d$ and $V=\mathbb E_{d, 1}.$ Let $\sigma : \mathbb C^d \rar \mathbb C^d$ be as defined in \eqref{sima-new}. 
%the involution given by 
%\beqn
%\sigma(z)=(z', -z_d), \quad z=(z', z_d) \in \mathbb C^{d-1} \times \mathbb C = \mathbb C^d.
%\eeqn
%Note that $J_\sigma = -1.$
Consider the $2$-proper map $\phi=(\phi_1, \ldots, \phi_d)$ as given in \eqref{phi-new}. 
%\beqn
%\phi_j(z)=z_j, ~j=1, \ldots, d-1, ~\phi_d(z)=z^2_d, \quad z=(z_1, \ldots, z_d).
%\eeqn 
%By \cite[Theorem~5.1]{R1982}, $\phi$ is $2$-proper. 
Clearly, $\phi$ maps $U$ onto $V$ and it is $\sigma$-invariant. Applying Theorem~\ref{main-result} to $\mathscr H(U):=H^2(U)$ with above choices of $\phi$ and $\sigma,$ we obtain the reproducing kernel Hilbert space $\mathscr H_\phi(V).$ Moreover, $\varGamma_\phi$ maps $\mathscr H_\phi(V)$ isometrically onto $\mathscr H^{\sigma}_-(U).$ We check that the reproducing kernel of 
$\mathscr H_\phi(V)$ 
is given by 
$$
\kappa_{\phi}(z, w) = \frac{1}{2}\frac{1-\inp{z'}{w'}_{_{\mathbb C^{d-1}}}}{[(1-\inp{z'}{w'}_{_{\mathbb C^{d-1}}})^2-z_d\overline{w}_d]^2}, \quad z=(z', z_d), w=(w', w_d) \in V.
$$
For simplicity, we verify this in the case $d=2.$ Note that 
\begin{align*}
   \kappa_-^{\sigma}(z, w) &= \frac{1}{2}\left[\frac{1}{(1-z_1\overline{w}_1-z_2\overline{w}_2)^2}-\frac{1}{(1-z_1\overline{w}_1+z_2\overline{w}_2)^2}\right] \\&=  \frac{2(1-z_1\overline{w}_1)z_2\overline{w}_2}{(1-z_1\overline{w}_1-z_2\overline{w}_2)^2(1-z_1\overline{w}_1+z_2\overline{w}_2)^2} \\ &= \frac{2(1-z_1\overline{w}_1)z_2\overline{w}_2}{[(1-z_1\overline{w}_1)^2-z_2^2\overline{w}_2^2]^2}.
\end{align*}
Since $J_{\phi}(z)=2z_2,$ by \eqref{formula-ii}, for $z, w \in U,$
\begin{align*}
 (2z_2)\kappa_{\phi}(\phi(z), \phi(w)) (2\overline{w}_2) = \frac{2(1-z_1\overline{w}_1)z_2\overline{w}_2}{[(1-z_1\overline{w}_1)^2-z_2^2\overline{w}_2^2]^2},
\end{align*}
which implies that
\beqn
2\kappa_{\phi}(\phi(z), \phi(w)) = \frac{1-\phi_1(z)\overline{\phi_1(w)}}{[(1-\phi_1(z)\overline{\phi_1(w)})^2-\phi_2(z)\overline{\phi_2(w)}]^2}.
 \eeqn
 It follows that
 \beqn
\kappa_{\phi}(z, w) = \frac{1}{2}\frac{1-z_1\overline{w}_1}{[(1-z_1\overline{w}_1)^2-z_2\overline{w}_2]^2}, \quad z, w \in V.
 \eeqn
We refer to $\mathscr H_\phi(V)$ as the {\it Hardy space $H^2(\mathbb E_{d, 1})$} of the egg domain $\mathbb E_{d, 1}.$ This may be justified in view of the following proposition.
\begin{proposition}
For $f \in H^2(\mathbb E_{d, 1}),$ the norm of $f$ is given by 
\beqn
\|f\|^2_{\phi} = \displaystyle \sup_{0 \leq r < 1} \int\limits_{\partial \mathbb E_{d, 1}} |f(rz', r^2z_d)|^2 \, dm(z),
\eeqn
where $m$ denotes the normalized Lebesgue measure on $\partial \mathbb E_{d, 1}.$ 
\end{proposition}
\begin{proof} Note that 
\beqn
\|f\|^2_{\phi} &=&  \sup_{0 \leq r < 1} \int_{\partial \mathbb B^d} |f(rz', r^2z_d^2)|^2 |J_{\phi}(rz)|^2\, d\sigma(z).
\eeqn
Now after change of variables $z' \mapsto z'$ and $z_d^2 \mapsto z_d,$ we get $|J_{\phi}(z)|^2\, d\sigma(z)=dm(z).$ This gives the required formula. 
%Hence, we get
%\beqn
%\|f\|^2_{\phi} = \displaystyle \sup_{0 \leq r < 1}r^2 \int\limits_{\partial \mathbb E_{2, 1}} |f(rz_1, r^2z_2)|^2 \, d\mu(z). 
%\eeqn
\end{proof}
As a consequence of the proposition above, one can see that the multiplier algebra of $H^2(\mathbb E_{d, 1})$ is $H^{\infty}(\mathbb E_{d, 1})$ with equality of norms. Consequently, one may apply \cite[Theorem~2.1]{CJP2022} to obtain a variant of von Neumann's inequality for $\mathbb E_{d, 1}.$ 

\subsection{Concluding remarks}

The main result of this paper (see Theorem~\ref{main-result}) provides a general scheme for constructing a functional Hilbert space from a known one provided the space and domain in question have symmetries in appropriate sense. Although computing kernels involved could be computationally challenging (cf. \cite{BFS1999, GGLV2021, T2013}), this transference principle is theoretically applicable in many situations. For instance, one can apply this to the Bergman space of any of the domains considered in Section~\ref{S3}. 
We leave the details to the reader. It is worth noting that Theorem~\ref{main-result} is applicable in case $U$ and $V$ are unbounded domains.  Recall that the {\it Segal-Bargmann space} $\mathscr H(\mathbb C^d)$ is the Hilbert space of holomorphic functions $f$ in $d$ complex variables such that 
\beqn
    \|f\|^2 := \int_{\mathbb C^d} |f(z)|^2 \exp(-|z|^2) dz < \infty.
\eeqn
It is a reproducing kernel Hilbert space with the kernel
$\kappa(z, w)=\exp(\inp{z}{w}),$ $z, w \in \mathbb C^d$ (see \cite[Subsection~4.4.5]{PR2016}). 
Let $U=V=\mathbb C^d.$
%Clearly, $\phi$ maps $U$ onto $V$ and it is $\sigma$-invariant. 
Applying Theorem~\ref{main-result} to $\mathscr H(\mathbb C^d)$ with choices of $\phi$ and $\sigma$ as defined in \eqref{sima-new} and \eqref{phi-new}, respectively, we obtain the reproducing kernel Hilbert space $\mathscr H_\phi(V).$ As seen in earlier examples, one may check that the reproducing kernel of 
$\mathscr H_\phi(V)$ 
is given by 
$$
\kappa_{\phi}(z, w) = \frac{1}{4}\exp\big(\inp{z'}{w'}_{_{\mathbb C^{d-1}}}\big) \sum_{n =0}^{\infty} \frac{(z_d \overline{w}_d)^n}{(2n+1)!}, \quad z=(z', z_d), w=(w', w_d) \in V.
$$

%\vskip.2cm

{\it Acknowledgment.} We convey our sincere thanks to Gargi Ghosh for a helpful discussion on the subject matter of this paper. We are also thankful to anonymous referee for pointing out certain gaps in the original draft and also providing a proof of the reproducing kernel property of $\mathscr H_\phi(V)$ as included in the first paragraph of the Proof of Theorem~\ref{main-result}(ii). 

{}

\end{document}